\documentclass[12pt]{amsart}
\usepackage{amsmath}
\usepackage{amssymb}

\newtheorem{theorem}{Theorem}[section]
\newtheorem{lemma}[theorem]{Lemma}
\newtheorem{proposition}[theorem]{Proposition}

\newtheorem{remark}[theorem]{Remark}
\newtheorem{conclusion}[theorem]{Conclusion}

\newcommand{\mP}{{\mathbb P}}
\newcommand{\A}{{\mathbb A}}
\newcommand{\C}{{\mathbb C}}

\newcommand{\G}{{\mathbb G}}

\newcommand{\Z}{{\mathbb Z}}
\newcommand{\cI}{{\mathcal{I}}}
\newcommand{\cF}{{\mathcal{F}}}
\newcommand{\HH}{{\mathbf H}}
\newcommand{\cX}{{\mathcal{X}}}
\newcommand{\Grass}{\operatorname{Grass}}
\newcommand{\Hilb}{\operatorname{Hilb}}

\newcommand{\tJ}{\tilde{J}}
\newcommand{\cA}{{\mathcal{A}}}
\newcommand{\cB}{{\mathcal{B}}}
\newcommand{\cO}{{\mathcal{O}}}
\newcommand{\cP}{{\mathcal{P}}}
\newcommand{\red}{\operatorname{red}}
\newcommand{\reg}{\operatorname{reg}}
\newcommand{\Pic}{\operatorname{Pic}}
\newcommand{\Proj}{\operatorname{Proj}}

\numberwithin{equation}{section}

\begin{document}

\baselineskip=16pt

\title[The irreducible components of $\Hilb^{4n}(\mP^3)$]{The irreducible components of $\Hilb^{4n}(\mP^3)$}

\author[G. Gotzmann]{Gerd Gotzmann}

\address{Isselstr. 34, D-48431 Rheine}

\email{g.gotzmann@t-online.de}

\date{November 14, 2008}

\begin{abstract}
It is shown that $\Hilb^{4n}(\mP^3)$ has exactly two irreducible components.
\end{abstract}

\maketitle

{\sc Introduction.} Let $k$ be an algebraically closed field of characteristic 0 and let $\HH=H_{4,1}=\Hilb^{4n}(\mP_k^3)$ denote the Hilbert scheme parametrizing the curves of degree 4 and genus 1, i.e. the closed subschemes of $\mP^3$ with Hilbert polynomial $4n$. Vainsencher and Avritzer have given an explicit description of that component $H_{VA}$ of $\HH$, which contains the points corresponding to complete intersections of two quadrics, and which shows that this component is smooth. From a general theorem of Reeves and Stillman it follows that there is a unique component $H_{RS}$ of $\HH$, which contains the lexicographic point. The purpose of this paper is to prove the following

\noindent {\bf Theorem.} {\it
$\HH_{\red}$ consists of the two irreducible components $H_{VA}$ and $H_{RS}$ of dimension 16 and 23 respectively.}

This result is similar to a result of Piene and Schlessinger, who showed that $\Hilb^{3n+1}(\mP_k^3)$ consists of two irreducible components, too. They showed moreover that these components as well as their intersection are smooth. It would be very interesting to know, if the same is true for $\HH=H_{4,1}$, for then, as in the case of $H_{3,0}$, it would follow that $H^1(\HH,\cO_{\HH})=(0)$ and $\Pic(\HH)\simeq\Z^3$, if $k=\C$. (The computation of the cohomology of the structure sheaf of $\Hilb^d(\mP^2)$ and of its Picard group had been achieved by Fogarty and Iarrobino, respectively; see the introductions of [F1] and [F2].)

\section{The Borel ideals in $\HH$}

A curve $C\subset\mP^3=\Proj P,P=k[x,y,z,t]$, of degree 4 and genus 1 is defined by an ideal $\cI\subset\cO_{\mP^3}$ with Hilbert polynomial $Q(T)={T-1+3\choose 3}+{T-5+2\choose 2}+{T-6+1\choose 1}$. There are 4 such ideals $\cB_i, 3\le i\le 6$, which are invariant under the action of the Borel group $B(4,k)$ on $P$, and which we call Borel ideals. The regularity of $\cB_i$ is $i$ and $\varphi_i(n)=h^0(\cB_i(n))$ denotes the Hilbert function:\\
\phantom{which are invariant}$\cB_3=(x^2,xy,y^3),\varphi_3=(0,0,2,8,19,36,60,92,\cdots)$\\
\phantom{which are invariant}$\cB_4=(x^2,xy,xz^2,y^4),\varphi_4=(0,0,2,8,19,36,\cdots)$\\
\phantom{which are invariant}$\cB_5=(x^2,xy,xz,y^5,y^4z),\varphi_5=(0,0,3,9,19,36,\cdots)$\\
\phantom{which are invariant}$\cB_6=(x,y^5,y^4z^2),\varphi_6=(0,1,4,10,20,36,\cdots)$

\section{The stratification of $\HH$ by means of the regularity}
A general theorem of Galligo says that for each ideal $\cI$ and general element $g$ of $GL(4,k)$ the initial ideal of $g(\cI)$ (in the inverse lexicographic order) is a Borel ideal. A theorem of Bayer and Stillman says that $\cI$ and the generic initial ideal of $\cI$ have the same regularity and the same Hilbert function. (For all these general results, cf.[Gre].)

Let $R_i$ be the locally closed set of ideals in $\HH$ with regularity $i$. We take the $R_i$ as subschemes of $\HH$ with the induced reduced scheme structure, and we consider each time an ideal $\cI\in R_i$ and put $I_n=H^0(\mP^3,\cI(n))$ and $I=\mathop{\oplus}\limits_{n\ge 0}I_n$.

\subsection{Description of $R_3$}
>From the Hilbert function $\varphi_3$ we see that $I_2=\langle f,g\rangle$ with two forms $f,g\in P_2$. If $f$ and $g$ have no common divisor, then it is well known that the ideal $(f,g)$ has the Hilbert function $\varphi_3$. If we suppose that $f$ and $g$ have a common divisor, we can write $f=\ell\ell_1,g=\ell\ell_1$ where $\ell,\ell_1,\ell_2,\cdots$ always denote nonzero linear forms in $P$. As $\dim P_1I_2=7$, one has $I_3=P_1I_2+f\cdot k$. Here the form $f\in P_3$ is not divisible by $\ell$, because then $I\subset (\ell)$ would follow. (As $\reg (\cI)=3$, one has $P_nI_3=I_{n+3},n\ge 0$.) $I_4=P_2I_2+P_1f$ has the dimension 19, $\dim P_2I_2=16, \dim fP_1=4$, and therefore $\dim P_2I_2\cap fP_1=1$ follows. As $f$ is not divisible by $\ell$, it follows that $f\in (\ell_1,\ell_2)_3$.

\begin{proposition}
$\cI$ has the regularity 3 iff

1. $\cI=(f,g),f,g\in P_2$ without a common divisor, or

2. $\cI=(\ell\ell_1,\ell\ell_2,f)$, where $\ell_1,\ell_2$ are linear independent and $f\in(\ell_1,\ell_2)_3$ is not divisible by $\ell$.
\end{proposition}

\begin{proof}
It remains to show that in the second case $\cI$ has the Hilbert function $\varphi_3$. Because of $\ell(\ell_1,\ell_2)_{n-1}\cap fP_{n-3}=\ell fP_{n-4}$ one has $\dim I_n={n-2+3\choose 3}+{n-2+2\choose 2}+{n\choose 3}-{n-1\choose 3}=Q(n)=\varphi_3(n),n\ge 4$.
\end{proof}

\subsection{Description of $R_4$}
The same argumentation as in (2.1) shows that $I_2=\ell(\ell_1,\ell_2)_1$ and $I_3=\ell(\ell_1,\ell_2)_2+f\cdot k$. Suppose that $f$ is not divisible by $\ell$. If $f\notin (\ell_1,\ell_2)_3$ then $\ell(\ell_1,\ell_2)_3\cap fP_1=(0)$ and $I_4\supset \ell(\ell_1,\ell_2)_3+fP_1$ would have a dimension $\ge 20$, contradiction. If $f\in (\ell_1,\ell_2)_3$, then $I$ would contain the ideal $J:=(\ell\ell_1,\ell\ell_2,f)$ which according to Proposition 2.1 would have the Hilbert function $\varphi_3=\varphi_4$ and regularity 3, contradiction. It follows that $f=\ell q$ with $q\in P_2$, but $q\notin (\ell_1,\ell_2)_2$. The ideal $J:=\ell(\ell_1,\ell_2,q)$ has the Hilbert polynomial ${n-2+3\choose 3}+{n-2+2\choose 2}+{n-3+1\choose 1}$, which agrees with the Hilbert function $\varphi$ of $\tJ$ if $n\ge 2$. From $\varphi(4)=18$ it follows that $I_4=J_4+p\cdot k,p\in P_4$. From $I_5=J_5+pP_1,\dim I_5=36, \dim J_5=33$ it follows that $\dim J_5\cap pP_1=1$. Here $p$ can not be divisible by $\ell$, as otherwise $I\subset (\ell)$ would follow. Therefore $J_5\cap pP_1=\ell p\cdot k$ and thus $p\in (\ell_1,\ell_2,q)_4$ follows.

\begin{proposition}
$\cI$ has the regularity 4 iff $\cI=(\ell(\ell_1,\ell_2,q),p)$, where $q\in P_2,q\notin (\ell_1,\ell_2)_2, p\in (\ell_1,\ell_2,q)_4$ and $p\notin (\ell)$.
\end{proposition}

\begin{proof}
It remains to show that an ideal of this shape has the Hilbert function $\varphi_4$. As $J_n\cap pP_{n-4}=\ell pP_{n-5}$, it follows that $\dim I_n=\varphi (n)+{n-4+3\choose 3}-{n-5+3\choose 3}=Q(n)=\varphi_4(n)$ for $n\ge 4$.
\end{proof}

We determine some other properties of $R_4$. Put $q(n)={n-1+3\choose 3}+{n-1+2\choose 2}+{n-2+1\choose 1}$. Each ideal $\cI\subset\cO_{\mP^3}$ with Hilbert polynomial $q(n)$ is 2-regular, and thus has the form $(\ell_1,\ell_2,q),q\in P_2/ (\ell_1,\ell_2)_2$. The mapping $(\ell_1,\ell_2,q)\mapsto \langle\ell_1,\ell_2\rangle$ makes the Hilbert scheme $H_q$ into a projective bundle over $\Grass^2(P_1)$ with fibres isomorphic to $\mP^2$. Thus $H_q$ is a 6-dimensional variety. We consider the following morphisms:\\
Put $X:=\mP(P_1)\times H_q\times\Grass^{19}(P_4)$ and define $f_1:X\to \cX_1:=\Grass^{18}(P_4)\times\Grass^{19}(P_4)\times\Grass^{33}(P_4)$ by $(\ell,\cI,F)\mapsto (\ell H^0(\cI(3)),F,H^0(\cI(4)))$ and $f_2:X\to\cX_2:=\Grass^{18}(P_4)\times\Grass^{19}(P_4)\times\Grass^{20}(P_4)$ by $(\ell,\cI,F)\mapsto (\ell H^0(\cI(3)),F,\ell P_3)$. Put $\cF_i:=\{(F_1,F_2,F_3)\in\cX_i|F_1\subset F_2\subset F_3\}$, and $X_i:=f_i^{-1}(\cF_i), i=1,2$. Finally define $\pi_i:X_i\to\mP(P_1)\times H_q$ by $(\ell,\cI,F)\mapsto (\ell,\cI)$. Then $\pi_i$ is surjective with fibres isomorphic to $\mP((\ell_1,\ell_2,q)_4/\ell(\ell_1,\ell_2,q)_3)$ $\simeq\mP^{14}$ and to $\mP(\ell P_3/\ell (\ell_1,\ell_2,q)_3)\simeq\mP^1$, respectively.

One sees that $(\ell(\ell_1,\ell_2,q),p)\mapsto (\ell,(\ell_1,\ell_2,q),\ell (\ell_1,\ell_2,q)_3+p  \cdot k)$ defines an isomorphism $R_4\to X_1 \setminus X_2$ and one has:

\begin{proposition}
$R_4$ is a smooth quasiprojective variety of dimension 23.
\end{proposition}

\subsection{Description of $R_5$}
One has $I_2=\langle f_1,f_2,f_3\rangle$. If any two of the $f_i\in P_2$ would not have a common divisor, then certainly $\cI\notin R_5$. If every pair of these forms, but not all of them, would have a common divisor, then $\cI$ would define a twisted cubic in $\mP^3$. Thus $I_2=\ell\cdot\langle \ell_1,\ell_2,\ell_3\rangle$, and $\ell_1,\ell_2,\ell_3$ are linear independent. If $L:=(\ell_1,\ell_2,\ell_3)$ then from $\dim (\ell L)_5=34$ it follows that $I_5=\ell L_4+\langle f,g\rangle$.

Suppose that $f\notin (\ell)$ and $f\notin L$. Then $fP_1\cap \ell L_5=(0)$ and $fP_2\cap\ell L_6=f\ell L_1$. It follows that:

\begin{equation}\label{1}
  \dim (\ell L_5+ fP_1) = 59
\end{equation}

\begin{equation}\label{2}
  \dim (\ell L_6+ fP_2)=90
\end{equation}

Because of $Q(6)=60$, from (2.1) it follows that

\begin{equation}\label{3}
  \dim [(\ell L_5+ fP_1)\cap gP_1] = 3.
\end{equation}

By a theorem of Macaulay (cf. [Gre]) we deduce from (2.3) that $\dim [(\ell L_6+fP_2)\cap gP_2]\ge 3^{(3)}=9$. But then from (2.2) it follows that
$$
  \dim (\ell L_6+fP_2+gP_2)\le 90+{5\choose 3}-9=91
$$
contradiction, as $Q(7)=92$.

\begin{conclusion}\label{2.4}
One has $f\in (\ell)$ or $f\in L$.
\end{conclusion}

Suppose that $f\in (\ell)$. From Conclusion 2.4 it follows that $g\in L$. One has $f=\ell p,p\in P_4$ but $p\notin L$. Then

\begin{equation}\label{4}
  \dim (\ell L_5\cap fP_1)=3.
\end{equation}

Because of $g\notin (\ell)$ one has $\dim \ell L_5\cap gP_1=1$, and it follows that

\begin{equation}\label{5}
  \dim (\ell L_5+gP_1)=58.
\end{equation}

Because of $\dim (\ell L_5+fP_1+gP_1)=Q(6)=60$ one has $\dim [(\ell L_5+gP_1)\cap fP_1]=2$, contrary to (2.4).

\begin{conclusion}\label{2}
One has $f\notin (\ell),g\notin (\ell)$ but $f\in L$ and $g\in L$.
\end{conclusion}

Because of $\dim I_6=60, \dim \ell L_5=55$ and $\ell (f,g)\subset\ell L$ one gets:

\begin{conclusion}\label{3}
There are linear forms $r_1$ and $r_2$, which are not divisible by $\ell$, such that $fr_1+gr_2\in \ell L$.
\end{conclusion}

By applying a linear transformation one can achieve one of the following cases:\\[2mm]
{\it 1st case:} $\ell=x,L=(y,z,t)$\\
We put $S=k[y,z,t]$. Because of $f,g\in L$ one can assume without restriction that $f,g\in S_5$ and $r_1,r_2\in S_1$. Then $fr_1+gr_2\in xL\iff fr_1+gr_2=0\Rightarrow f=\ell_1h,g=\ell_2h$, where $h\in S_4$ and $\ell_1,\ell_2\in S_1$ are linear independent.\\[2mm]
{\it 2nd case:} $\ell=x,L=(x,y,z)$\\
Again $S:=k[y,z,t]$ and one can assume without restriction that $f=f'+\alpha xt^4,f'\in S_5,\alpha\in k$ and $g\in S_5$. Moreover one can assume that in the relation

\begin{equation}\label{6}
  fr_1+gr_2\in x(x,y,z)
\end{equation}

one has $r_1,r_2\in S_1$. Then (2.6) is equivalent to

\begin{equation}\label{7}
  f'r_1+gr_2=0\quad\mbox{and}\quad \alpha r_1\in (x,y,z).
\end{equation}

We note that $r_1$ and $r_2$ are linear independent, because otherwise $f'$ and $g$ would be linear dependent, and by linearly combining $f$ and $g$, we could achieve that $f=xt^4$, contradicting Conclusion 2.5. It follows that $f=\ell_1 h+\alpha xt^4,g=\ell_2h,h\in S_4, \ell_1,\ell_2\in S_1$ linear independent and $\alpha=0$ or $\ell_2\in\langle y,z\rangle$.

We now suppose that $f$ and $g$ have the shape described in the first or second case. Then $r=(\ell_2,-\ell_1)$ is a relation between $f$ and $g$ (modulo $xL$) in any case. One sees that any such relation has the form $p\cdot r,p\in S$. We conclude that $\dim I_n=Q(n),n\ge 5$.

We now determine the dimension of $R_5$. Put $q(n)={n-1+3\choose 3}+{n-5+2\choose 2}+{n-5+1\choose 1}$, and let $H_q$ be the Hilbert scheme of ideals with Hilbert polynomial $q(n)$, i.e. $H_q=H_{4,2}$. Define $\pi:R_5\to H_q$ by $(\ell L,f,g)\mapsto (\ell,f,g)$. From the preceding argumentation one sees that the fibre of $\pi$ over $(\ell,f,g)$ is equal to $\mP^3\setminus V_+(\ell)$ union a subset of $V_+(\ell)\times\A^1$, depending on $(\ell,f,g)$. In any case, all fibres are 3-dimensional. The image of $\pi$ is equal to the closed subset $H_1$ of $H_q$ consisting of the ideals $(\ell, Lf)$, where $\ell\in P_1,\langle\ell\rangle \subset L$ a 3-dimensional subspace of $P_1$ and $f\in P_4/\ell P_3$. One sees that $\dim H_1=3+2+{4+2\choose 2}-1=19$ and we get
$$
  \dim R_5=22.
$$

\subsection{Description of $R_6$} 
One sees easily that $R_6$ is closed in $\HH$ and consists of the ideals $(\ell, f(h,g))$ where $\ell\in P_1,f\in P_4/\ell P_3,h\in P_1/\ell\cdot k$ and $g\in P_2/(\ell,h)_2$. It follows that
$$
  \dim R_6=21.
$$

\section{The Reeves-Stillman component of $\HH$}\label{3}

One considers ideals of the form $\cI=(\ell,f)\cap\cP_1\cap\cP_2,\ell\in P_1,f\in P_4/\ell P_3,\cP_1$ and $\cP_2$ prime ideals in $P$, which define two different closed points in $\mP^3 \setminus (V_+(\ell)\cap V_+(f))$. Then $\cI$ has the Hilbert polynomial $Q$ and the set of such ideals, i.e. the set of corresponding points in $\HH(k)$ is denoted by $\cF$. Then a beautiful theorem of Reeves and Stillman ([RS], thm. 1.4) applied to our special situation gives:\\
The closure of $\cF$ in $\HH$ is an irreducible component $Z$ of $\HH$, which we also denote by $H_{RS}$. The lexicographic point $z$, which corresponds to the ideal $\cB_6$, lies in $Z$ and $z$ is a smooth point (and therefore does not lie in any other component of $\HH$). The dimension formula in ([RS], thm. 4.1), or more easily, a dimension count shows that $\dim Z=23$.

It is clear that $Z$ is invariant under $G=GL(4,k)$ and there is a maximal open subset $Y$ of $\HH$ which is contained in $Z$ and which contains the point $z$.

In order to show that $R_5\cup R_6\subset Z$ we need some simple facts.

\begin{remark}\label{1}
If $p\in\HH(k)$ and $\overline{Gp}\cap R_6\neq\emptyset$, then $p\in Y$.
\end{remark}

\begin{proof}
There is a closed point in $\overline{Gp}\cap R_6$ which is $B(4,k)$-invariant and necessarily this point is equal to $z$. There is a $g\in G$ such that $gp\in Y$ and $p\in g^{-1}(Y)=Y$ follows.
\end{proof}

\begin{remark}\label{2}
$R_6\subset Y$
\end{remark}

\begin{remark}\label{3}
$R_5\subset\overline{Y}$
\end{remark}

\noindent {\it Proof.} To begin with, let the point $p\in R_5$ correspond to an ideal $\cI$ as in the first case of (2.3). By applying a suitable linear transformation one can achieve that $\cI=(t(x,y,z),f,g)$. As $f,g\in (x,y,z)$ one can assume without restriction that $f,g\in S_5$, where $S=k[x,y,z]$. By applying a linear transformation which involves only the variables $x,y,z$ one can achieve that the initial monomials of $f$ and $g$ generate a 2-dimensional, $B(3,k)$-invariant subspace of $S_5$. It follows that the initial monomials of $f$ and $g$ are $x^5$ and $x^4y$, respectively. Thus we can write
$$
  f=x^5+x^4f_1+x^3f_2+x^2f_3+xf_4+f_5
$$
$$
  g=x^4g_1+x^3g_2+x^2g_3+xg_4+g_5
$$
where $f_i,g_i\in k[y,z]_i$ and $g_1=y+az,a\in k$.

Let $\G_a$ operate on $P$ by $\psi_{\alpha}:t\mapsto -\alpha x+t$ and as identity on the remaining variables. From $V=V_{\alpha}:=(-\alpha x+t)\langle x,y,z\rangle\subset\psi_{\alpha}(H^0(\cI(2)))$ it follows that
$$
  tx\equiv \alpha x^2,\; ty\equiv \alpha xy,\; tz\equiv\alpha xz
$$
modulo $V$. One deduces from this that 

\begin{equation}\label{8}
  x^{i+1}\equiv\alpha^{-i}t^ix, x^iy\equiv\alpha^{-i}t^iy,x^iz\equiv\alpha^{-i}t^iz
\end{equation}

modulo $P_{i-1}\cdot V,1\le i\le 4$, where $\alpha\in k^*$. Clearly $f$ and $g$ are invariant under this $\G_a$-operation. By using (3.1) it follows that
$$
  \begin{array}{l}
  f\equiv\alpha^{-4}t^4x+\alpha^{-4}t^4f_1+\cdots +\alpha^{-1}tf_4+f_5\\
  g\equiv \alpha^{-4}t^4g_1+\alpha^{-3}t^3g_2+\cdots+\alpha^{-1}tg_4+g_5
  \end{array}
$$

modulo $P_3V\subset\psi_{\alpha}(H^0(\cI(5)))$. Now, if $\G_m$ operates on $P$ by
$$
  \sigma(\lambda):x\mapsto x,y\mapsto\lambda^{-2}y,z\mapsto\lambda^{-2}z,t\mapsto t,
$$
then $V$ is invariant under this operation and one has

\begin{equation}\label{3.2}
  \begin{array}{l}
  \sigma(\alpha)f\equiv\alpha^{-4}t^4x+\alpha^{-6}t^4f_1+\ldots +\alpha^{-9}tf_4+\alpha^{-10}f_5\\
  \sigma(\alpha)g\equiv\alpha^{-6}t^4g_1+\cdots +\alpha^{-9}tg_4+\alpha^{-10}g_5
  \end{array}
\end{equation}

if $\alpha\in k^*$.

The morphism $\tau:\A^1\setminus \{0\}\to\HH$ defined by $\alpha\mapsto\cI_{\alpha}:=\psi_{\alpha}\circ\sigma(\alpha)\cI$ has an extension $\overline{\tau}:\mP^1\to\HH$. Let $\cI_{\infty}=\lim\limits_{\alpha\to\infty}\cI_{\alpha}=\lim\limits_{\alpha\to 0}\cI_{1/\alpha}$ be the ``limit ideal'', which corresponds to the point $\overline{\tau}(\infty)\in\HH(k)$. From (3.2) it follows that
$$
  x\langle x,y,z\rangle\cdot P_3\oplus t^4x\cdot k\oplus t^4g_1\cdot k\subset H^0(\cI_{\infty}(5))
$$
and thus $x(x,y,z,t^4)\subset\cI_{\infty}$. But then $x\in\cI_{\infty}$ and the Hilbert function of $\cI_{\infty}$ is greater than the Hilbert function $\varphi_5$ of $\cI$. This shows that $\cI_{\infty}\in R_6$, and by Remark (3.2) $p\in Y$ follows.

Now let the point $p$ correspond to an ideal $\cI$ as in the second case of (2.3). Again one can assume that $\cI=(x(x,y,z),f,g)$. Using the same notations as in (2.3) we can write $f=\ell_1h+\alpha xt^4,g=\ell_2h,\ell_1,\ell_2,h\in k[y,z,t]$.\\
{\it Subcase 1:} $\alpha\neq 0$\\
Then we let $\G_m$ operate by $\sigma(\lambda):x\mapsto \lambda x$ and as the identity on the remaining variables. Then we get $\cI_{\infty}:=\lim\limits_{\lambda\to\infty}\sigma(\lambda)\cI\supset x(x,y,z,t^4)$ and just like before $p\in Y$ follows.\\
{\it Subcase 2:} $\alpha=0$\\
By replacing one of the forms $f$ and $g$ by a suitable linear combination we can assume without restriction that $\ell_2\in\langle y,z\rangle$. If we put $f_1=\ell_1h+xt$, then $(\ell_2,-\ell_2)$ is a relation between $f_1$ and $g$ modulo $x(x,y,z)$. But then the ideal $\cI_1=(x(x,y,z),f_1,g)$ is as in Subcase 1, and we have shown already that $\cI_1\in Y$ follows. As $\cI=\lim\limits_{\lambda\to 0}\sigma(\lambda)\cI_1$, we conclude that $p\in\overline{Y}$.\hfill $\Box$

\begin{lemma}\label{3.4}
$R_5\cup R_6\subset\overline{Y}=H_{RS}$.
\end{lemma}

\section{The Vainsencher-Avritzer component of $\HH$}\label{4}
The set of all curves in $\mP^3$, which are complete intersections of two quadrics is an open set $V\subset R_3$, which is isomorphic to an open subset of $\Grass^2(P_2)$. Thus $V$ is irreducible of dimension 16. The closure $\overline{V}$ in $\HH$ is an irreducible component of $\HH$, which has been described by Vainsencher and Avritzer and which we denote by $H_{VA}$.

Using a similar argumentation as in (2.2) one can show that $R'_3:=R_3\setminus V$ is a smooth variety of dimension 15. We want to show that $R'_3\subset\overline{V}$. (Probably this can be read from the diagram on page 48 in [VA].).

Let be $\cI=(\ell(\ell_1,\ell_2),F)\in R'_3$ and write $F=\ell_1p+\ell_2q,p,q\in P_2$ not divisible by $\ell$ (cf. Proposition 2.1). We put $f_{\alpha}:=\ell\ell_1+\alpha q,g_{\alpha}=\ell\ell_2-\alpha p$, where $\alpha$ is a variable. If $\tau\in P$ is a divisor of $f_{\alpha}$ and $g_{\alpha}$ for an element $\alpha  \in k^*$, then $\tau$ is a divisor of $\ell_2f_{\alpha}-\ell_1g_{\alpha}=\alpha F$. It follows that $\tau$ is an element of the finite set of divisors of $F$ (modulo a constant factor, of course). Suppose there are infinite many $\alpha\in k$ such that $f_{\alpha}$ and $g_{\alpha}$ have a common divisor. Then there is a form $\tau\in P$ of positive degree such that $\cA:=\{\alpha\in k|\tau$ divides $f_{\alpha}$ and $g_{\alpha}\}$ is an infinite set. If $\alpha,\beta\in\cA,\alpha\neq\beta$, then $\tau$ is a divisor of $f_{\alpha}-f_{\beta}=(\alpha-\beta)q$ and of $g_{\alpha}-g_{\beta}=(\beta-\alpha)p$. It follows that $\tau$ divides $\ell\ell_1,\ell\ell_2,p,q$. As $\ell_1$ and $\ell_2$ are linear independent, it follows that $\tau=\ell$, contradiction. It follows that $A:=\{\alpha\in k|f_{\alpha}$ and $g_{\alpha}$ have a common divisor$\}$ is a finite set and the morphism $\varphi:\A^1\setminus A\to V\hookrightarrow\HH$ defined by $\alpha\mapsto\cI_{\alpha}:=(f_{\alpha},g_{\alpha})$ has an extension $\overline{\varphi}:\A^1\to\HH$. From $\ell_2f_{\alpha}-\ell_1g_{\alpha}=\alpha F$ it follows $F\in\cI_{\alpha}$ for all $\alpha\in\A^1/A$. If $\cI_0$ denotes the limit ideal $\lim\limits_{\alpha\to 0}\cI_{\alpha}$, i.e. the ideal corresponding to $\overline{\varphi}(0)\in\HH(k)$, one sees that $(\ell\ell_1,\ell\ell_2,F)\subset\cI_0$. As both ideals have the same Hilbert polynomial $(\ell\ell_1,\ell\ell_2,F)=\cI_0$ follows. We get:

\begin{lemma}\label{2}
The set of 3-regular ideals, which are not complete intersections, is contained in $H_{VA}$.
\end{lemma}

\section{Proof of the theorem}\label{5}
One has $\HH=R_3\mathop{\cup}\limits^{\cdot} (\overline{R}_4\cup\overline{R}_5\cup R_6)$ and we showed in the sections (2.2)--(2.4) and 4 that the dimensions of the subschemes are 16, 23, 22, 21, respectively. As $H_{RS}$ is an irreducible component of dimension 23 containing $R_5\cup R_6$ (Lemma 3.4), we deduce from Proposition 2.3 that $H_{RS}=\overline{R}_4$, and Lemma 4.1 gives finally $\HH=H_{VA}\cup H_{RS}$.


\end{document}